\documentclass[]{article}

\usepackage{makeidx}  

\usepackage{amsthm} 

\usepackage{latexsym}
\usepackage[pdftex]{graphics}
\usepackage{xypic}
\usepackage{amsmath}
\usepackage{amssymb}
\xyoption{all} 
\usepackage{mathrsfs}
\usepackage{eufrak}

\usepackage[hmargin=3.5cm,vmargin=3.5cm]{geometry}

\newtheorem{theorem}{Theorem}[section]

\newtheorem{lemma}[theorem]{Lemma}

\newtheorem{definition}[theorem]{Definition}

\newtheorem{proposition}[theorem]{Proposition}

\newtheorem{corollary}[theorem]{Corollary}

\newtheorem{remark}[theorem]{Remark}

\newcommand{\C}{{\mathcal C}}

\newcommand{\F}{{\mathcal F}}

\newcommand{\pIso}{\bf pIso}
\newcommand{\PI}{\bf pIso}

\newcommand{\dagg}{(\ )^\dagger}

\newcommand{\decode}{\rhd}
\newcommand{\code}{\lhd}

\newcommand{\dc}{{\lhd\hspace{-0.148em}\rhd}}

\begin{document}
%
%
%
\title{Classical structures based on unitaries}
%
%
\author{Peter Hines}
%
%
%

\maketitle              

\begin{abstract}
Starting from the observation that distinct notions of copying have arisen in different categorical fields (logic and computation, contrasted with quantum mechanics ) this paper addresses the question of when, or whether, they may coincide. 

Provided all definitions are strict in the categorical sense, we show that this can never be the case. However, allowing for the defining axioms to be taken up to canonical isomorphism,  a close connection between the  {\em classical structures} of categorical quantum mechanics, and the categorical property of {\em self-similarity} familiar from logical and computational models 
 becomes apparent. 
 
The required canonical isomorphisms are non-trivial, and mix both typed (multi-object) and untyped (single-object) tensors and structural isomorphisms; we give coherence results that justify this approach. 

We then give a class of examples where distinct self-similar structures at an object determine distinct matrix representations of arrows, in the same way as classical structures determine matrix representations in Hilbert space. We also give analogues of familiar notions from linear algebra in this setting such as changes of basis, and diagonalisation. 

\end{abstract}

%
\section{Introduction}
This paper addresses the question of whether the {\em classical structures} used in categorical quantum mechanics (based on monoid co-monoid pairs with additional structure) can ever be built from unitary maps --- can the monoid and co-monoid arrows be mutually inverse unitaries? From a simplistic perspective, the answer is negative (Corollary  \ref{notypedassoc} and Corollary \ref{nodoublecc}); however when we allow the defining conditions of a classical structure to be taken {\em up to canonical isomorphism}, not only is this possible, but the required conditions (at least, using the redefinition of Abramsky and Heunen \cite{AH}) may be satisfied by any pair of mutually inverse unitaries (Theorem \ref{AHssFull}) with the correct typing in a $\dagger$ monoidal category.

However, the required canonical isomorphisms are non-trivial, and mix `typed' and `untyped' (i.e. multi-object and single-object) monoidal tensors and canonical isomorphisms. We study these, and in Appendix \ref{simplecoherence} demonstrate how the question of coherence in such a setting may  be reduced to the well-established coherence results found in \cite{MCL}. 

We  illustrate this connection with a concrete example, and show how in this setting, self-similar structures play an identical role to that played by classical structures in finite-dimensional Hilbert space --- that of specifying and manipulating matrix representations. We also give analogues of notions such as `changes of basis' and `diagonalisation' in this setting. 

\section{Categorical preliminaries}
The general area of this paper is firmly within the field of $\dagger$ monoidal categories. However, due to the extremal settings we consider, we will frequently require monoidal categories without a unit object. We axiomatise these as follows:
\begin{definition}\label{semicat}
Let $\C$ be a category. We say that $\C$ is {\bf semi-monoidal} when there exists a {\bf tensor}
$( \_\otimes \_ ) : \C\times \C \rightarrow \C$
together with a natural indexed family of {\bf associativity isomorphisms}
$\{  \  \tau_{A,B,C} : A\otimes (B\otimes C)\rightarrow (A\otimes B) \otimes C \} _{A,B,C\in Ob(\C)}$
satisfying MacLane's {\bf pentagon condition} $( \tau_{A,B,C}\otimes 1_D)   \tau_{A,B\otimes C,D}  (1_A \otimes  \  \tau_{B,C,D}) =    \tau_{A\otimes B,C,D}  \tau_{A,B,C\otimes D}$. 

When there also exists a natural object-indexed natural family of {\bf symmetry isomorphisms} $\{ \sigma_{X,Y}:X\otimes Y\rightarrow Y\otimes X \}_{X,Y\in Ob(\C) }$
satisfying MacLane's {\bf hexagon condition} $\tau_{A,B,C}\sigma_{A\otimes B,C}  \tau_{A,B,C} =(\sigma_{A,C}\otimes1_B)  \tau_{A,C,B} (1_A\otimes \sigma_{B,C})$ 
we say that $(\C,\otimes ,  \tau,\sigma)$ is a {\bf symmetric semi-monoidal category}. 
A semi-monoidal category $(\C,\otimes, \  \tau_{\_,\_,\_})$ is called {\bf strictly associative} when $  \  \tau_{A,B,C}$ is an identity arrow\footnote{This is not implied by equality of objects $A\otimes (B\otimes C) = (A\otimes B)\otimes C$, for all $A,B,C\in Ob(\C)$. Although MacLane's pentagon condition is trivially satisfied by identity arrows, naturality with respect to the tensor may fail. Examples we present later in this paper illustrate this phenomenon.}, for all $A,B,C\in Ob(\C)$. 
A functor $\Gamma : \C \rightarrow \mathcal D$ between two semi-monoidal categories $(\C ,\otimes_\C )$ and $(\mathcal D,\otimes_\mathcal D)$ is called (strictly) {\bf semi-monoidal} when $\Gamma (f\otimes_\mathcal C g) = \Gamma(f) \otimes_\mathcal D \Gamma(g)$.
A semi-monoidal category $(\C,\otimes)$ is called {\bf monoidal} when there exists a {\bf unit object} $I\in Ob(\C)$, together with, for all objects $A\in Ob(\C)$, distinguished isomorphisms 
$\lambda_A:I \otimes A \rightarrow A$ and $\rho_A : A\otimes I \rightarrow A$
satisfying MacLane's {\bf triangle condition} $1_U\otimes \lambda_V = (\rho_U\otimes 1_V) \tau_{U,I,V}$ for all $U,V\in Ob(\C)$.

A {\bf dagger} on a category $\C$ is simply a duality that is the identity on objects; 
that is,  a contravariant endofunctor $( \ )^\dagger:\C\rightarrow \C$ satisfying $ (1_A)^\dagger = 1_A$ and $\left( (f)^\dagger\right)^\dagger =f$, for all $A\in Ob(\C)$ and $f\in \C(A,B)$.  
An arrow $U\in \C(X,Y)$ is called {\bf unitary} when it is an isomorphism with inverse given by $U^{-1}=U^\dagger\in \C(Y,X)$.

When $\C$ has a (semi-) monoidal tensor $\_ \otimes \_ :\C\times \C\rightarrow \C$, we say that $(C,\otimes )$ is $\dagger$ (semi-) monoidal when $(\ )^\dagger$ is a (semi-) monoidal functor, and all canonical isomorphisms are unitary.
\end{definition}

\begin{remark}{\bf Coherence for semi-monoidal categories} A close reading of \cite{MCL} will demonstrate that MacLane's coherence theorems for associativity and commutativity are equally applicable in the presence of absence of a unit object. 
The theory of Saavedra units \cite{JK} also demonstrates that the properties of the unit object are independent of other categorical properties (including associativity). Motivated by this, we give a simple method of adjoining a strict unit object to a semi-monoidal category that is left-inverse to the obvious forgetful functor.
\end{remark}

\begin{definition} Let $(\C,\otimes)$ be a semi-monoidal category. We define its {\bf unit augmentation} to be the monoidal category given by the following procedure:
We first take the coproduct of $\C$ with the trivial group $\{ 1_I\}$, considered as a single-object dagger category. We then extend the tensor of $\C$ to the whole of $\C\coprod I$ by taking $\_ \otimes I = Id_{\C\coprod I} = I \otimes \_$. 
\end{definition}
It is straightforward that the unit augmentation of a semi-monoidal category is a monoidal category; a full proof, should one be needed, is given as an appendix to \cite{PHarxiv}. Similarly, it is a triviality that if $(C,\otimes)$ is $\dagger$ semi-monoidal, then its unit augmentation is dagger monoidal.

The connection of the above procedure with MacLane's coherence theorems for associativity and commutativity should then be clear; any diagram that commutes in $\C$ also commutes in the unit augmentation; conversely any diagram (not containing the unit object) that commutes in the unit augmentation also commutes in $\C$. Thus MacLane's coherence theorems (with the obvious exclusion of the unit object) also hold in the semi-monoidal and unitless cases.

\section{Classical structures and their interpretation}\label{classdefs}
Classical structures were introduced in \cite{CP} as an abstract categorical interpretation of {\em orthonormal bases} in Hilbert spaces and the special role that these play in quantum mechanics (i.e. as sets of compatible disjoint measurement outcomes). This intuition was validated in \cite{CPV}, where it is proved that in the category of  finite-dimensional Hilbert spaces, there is a bijective correspondence between orthonormal bases and classical structures. 
Mathematically, classical structures are symmetric $\dagger$ Frobenius algebras in $\dagger$ monoidal categories satisfying a simple additional condition.
\begin{definition}\label{FA}
Let $(\mathcal C ,\otimes ,I, (\ )^\dagger )$ be a  strictly associative monoidal category. 
A {\bf Frobenius algebra} consists of a co-monoid structure $(\Delta :S\rightarrow S \otimes S, \top : S \rightarrow I)$ and a monoid structure $(\nabla:S\otimes S \rightarrow S,\bot :  I\rightarrow S)$ at the same object, where the monoid / comonoid pair satisfy the {\bf Frobenius condition} 
\[ (1_S\otimes \nabla) (\Delta\otimes 1_S) = \Delta \nabla = (\nabla \otimes 1_S)(1_S\otimes \Delta) \]
Expanding out the definitions of a monoid and a comonoid structure gives:
\begin{itemize}
\item {\bf (associativity)} $\nabla(1_S\otimes \nabla) = \nabla(\nabla\otimes 1_S)\in \mathcal C(S\otimes S \otimes S , S)$.
\item {\bf (co-associativity)} $(\Delta \otimes 1_S)\Delta = (1_S \otimes \Delta)\Delta  \in \mathcal C(S,S\otimes S \otimes S)$.
\item {\bf (unit)} $\nabla(\bot \otimes 1_S) = \nabla (1_S\otimes \bot)$.
\item {\bf (co-unit)} $(\top \otimes _S)\Delta = 1_X\otimes \top) \Delta$.
\end{itemize}
A Frobenius algebra $(S,\Delta,\nabla,\top,\bot)$ in a $\dagger$ monoidal category is called a {\bf dagger Frobenius algebra} when it satisfies $\Delta^\dagger=\nabla$ and $\top^\dagger=\bot$.

Let $(\C,\otimes)$ be a symmetric $\dagger$ monoidal category, with symmetry isomorphisms $\sigma_{X,Y}\in \C(X\otimes Y,Y\otimes X)$. A $\dagger$ Frobenius algebra is called {\bf commutative} when 
$\sigma_{S,S}\Delta = \Delta$, and hence $\nabla = \nabla\sigma_{S,S}$. 
A {\bf classical structure} is then a commutative $\dagger$ Frobenius algebra satisfying the following additional condition:
\begin{itemize}
\item {(\bf The classical structure condition)} $\Delta^\dagger$ is left-inverse to $\Delta$, so $\nabla\Delta=1_S$.
\end{itemize}
\end{definition}

\begin{remark}\label{CSascopying} The intuition behind a classical structure is that it describes related notions of copying and deleting (the comonoid and monoid structures). The underlying intuition is that, although arbitrary quantum states are subject to the no-cloning and no-deleting theorems \cite{NC,ND}, quantum states that are `classical' (i.e. members of some fixed orthonormal basis -- the `computational basis' of quantum computation) can indeed be both copied and deleted (against a copy) using the fan-out maps and their inverses \cite{FANOUT}. 

An aim of this paper is to compare such a notion of copying with a  distinct notion of copying that arose independently in models  of resource-sensitive logical and computational systems \cite{GOI,GOI3}, and to demonstrate connections, via the theory of untyped categorical coherence, between these notions.
\end{remark}

\subsection{Classical structures without units}
As noted in \cite{AH}, when considering the theory of classical structures in arbitrary separable Hilbert spaces, is often necessary to generalise Definition \ref{FA} to the setting where unit objects are not considered -- i.e. to lose the {\em unit} and {\em co-unit} axioms. We refer to \cite{AH} for a study of how much of the theory of \cite{CP} carries over to this more general setting, and give the following formal definition, which is a key definition of \cite{AH} in the strict, semi-monoidal setting:

\begin{definition}\label{AHFA}
Let $(\C,\otimes)$ be a strictly associative semi-monoidal category. 
An {\bf Abramsky-Heunen  (A.-H.) dagger Frobenius algebra} consists of a triple $(S\in Ob(\C),\Delta:S\rightarrow S\otimes S,\nabla=\Delta^\dagger: S\otimes S \rightarrow S)$  satisfying
\begin{enumerate}
\item {\bf (associativity)} $\nabla(1_S\otimes \nabla)=(1\otimes \nabla) \nabla\in \C(S\otimes S\otimes S,S)$.
\item {\bf (Frobenius condition)} $ \Delta \nabla =(1_S\otimes \nabla) (\Delta \otimes 1_S) \ \in \mathcal C(S\otimes S, S\otimes S)$
\end{enumerate}
An A-H $\dagger$ Frobenius algebra is an {\bf A-H classical structure} when $(\C,\otimes )$ is symmetric, and the following two conditions are satisfied:
\begin{enumerate}
\item[3.] {\bf (Classical structure condition)} $\nabla\Delta =1_S$,
\item[4.] {\bf (Commutativity)} $\sigma_{S,S} \Delta = \Delta$.
\end{enumerate}
\end{definition}

\subsection{Classical structures, and identities up to isomorphism}\label{uptoiso}
It is notable that the definitions of the previous sections are based on strictly associative tensors. Consider the definition presented of a monoid within a category, $(1_A\otimes \nabla)\nabla = (\nabla\otimes 1_A)\nabla$. 
Drawing this as a  commutative diagram 
\[ \xymatrix{
A \otimes A \ar[d]_{\nabla\otimes 1_A}			& A \ar[r]^\nabla \ar[l]_\nabla	&  A \otimes A \ar[d]^{1_A\otimes \nabla} \\
(A\otimes A)\otimes A 						&						& A\otimes (A\otimes A) \ar@{-}[ll]^{Id.}
}
\]
demonstrates that this definition relies on the identity of objects $A\otimes (A\otimes A) = (A\otimes A) \otimes A$ required for strict associativity\footnote{We emphasise that such identities of objects are a necessary, but not sufficient, condition for strict associativity of a tensor; see the footnote to Definition \ref{semicat}.}
 in an essential way. The definition of a co-monoid requires the same identification of objects.

\noindent
Similarly, the Frobenius condition $(1_A\otimes \nabla)(\Delta \otimes 1_A) = \Delta\nabla$
may be drawn as 
\[ 
\xymatrix{ 
A \otimes A \ar[r]^\nabla \ar[d]_{\Delta \otimes 1_A} 		& A \ar[r]^\Delta 	& A \otimes A \\
A\otimes A \otimes A \ar[rr]_{Id}						&				& A\otimes A \otimes A \ar[u]_{1_A\otimes \nabla}
}
\]

\begin{remark}

A significant feature of this paper is the relaxation of these strict identities, to allow the above definitions to be satisfied up to canonical isomorphisms.  
When making this generalisation, the choice of canonical isomorphisms seems to be straightforward enough; however, there are other possibilities. We take a more general view and  allow the axioms above to be satisfied up to any canonical isomorphisms for which there exists a suitable theory of coherence.
\end{remark}

\section{Self-similarity, and $\dagger$ self-similar structures}
By contrast with the strongly physical intuition behind classical structures, self-similar structures were introduced to study infinitary and type-free behaviour in logical and computational systems. Their definition is deceptively simple --  they are  simply a two-sided form of the `classical structure' condition of Definition \ref{FA}. The following definition is based on \cite{PH98,PH99}:
\begin{definition}\label{basicselfsim}
Let $(\C,\otimes)$ be a semi-monoidal category. A  {\bf self-similar structure} $(S,\code,\decode )$ is an object 
$S\in Ob(\C)$, together with two mutually inverse arrows
\begin{itemize}
\item {\bf (code)} $\code\in \C(S\otimes S, S)$.
\item {\bf (decode)}  $\decode \in \C(S,S\otimes S)$.
\end{itemize}
satisfying $\decode\code = 1_{S\otimes S}$ and $\code \decode= 1_S$. A {\bf dagger self-similar structure} is a self-similar structure in a $\dagger$ monoidal category with unitary code / decode arrows. 
\end{definition}

\begin{remark}
Recall from Remark \ref{CSascopying} the intuition of the classical structures of categorical quantum mechanics as a (restricted form of) copying and deleting that is applicable to computational basis states only. The very simple definition of a self-similar structure above is also clearly describing a notion of copying, albeit at the level of objects rather than arrows; simply, there are canonical arrows that provide isomorphisms between one copy of an object, and two copies of an object.  
A key theme of this paper is the relationship between these two notions of copying: whether the monoid / comonoid structure of an A.H. classical structure can also be a $\dagger$ self-similar structure, and whether a classical structure can also define a monoid / comonoid satisfying the Frobenius condition, \&c. 

Instead of a simple yes/no answer, we will observe a close connection with the theory of categorical coherence and strictification. In the strict case, requiring unitarity of the monoid / comonoid arrows implies a collapse to the unit object (Corollaries  \ref{notypedassoc} and \ref{nodoublecc}), whereas, up to a certain set of (non-trivial) canonical isomorphisms, $\dagger$ self-similar structures do indeed satisfy the conditions for an A.-H. classical structure (Theorems \ref{AHssFull} and \ref{AHssStrict}).
\end{remark}

We will first require many preliminary results on self-similar structures and their relationship with the theory of monoidal categories; we start by demonstrating that  $\dagger$ self-similar structures are unique up to unique unitary:
\begin{proposition} \label{unique}
Let $(S,\code,\decode )$ be a $\dagger$ self-similar structure of a $\dagger$ semi-monoidal category $(\mathcal C,\otimes, \dagg)$. Then 
\begin{enumerate}
\item  Given an arbitrary unitary $U\in \C(S,S)$, then $(S,U\code ,  \decode U^{\dagger})$ is also a $\dagger$ self-similar structure.
\item Given $\dagger$ self-similar structures $(S,\code ,\decode)$ and $(S,\code',\decode' )$,  there exists a unique unitary $U\in \mathcal C(S,S)$ such that 
$\code' = U\code \in \mathcal C(S\otimes S,S)$ and $\decode'=\decode U^{\dagger} \in \mathcal C(S,S\otimes S)$.
\end{enumerate}
\end{proposition}
\begin{proof}$ \ $ \\
\begin{enumerate}
\item Since $U$ is unitary, $U\code  \decode U^\dagger  = 1_S$ and $ \decode U^\dagger U \code =1_{S\otimes S}$. Thus, as the composite of unitaries is itself unitary, $(S,U\code ,  \decode U^{\dagger})$ is a $\dagger$ self-similar structure.
\item We define $U=\code'\decode\in \C(S,S)$, giving its inverse as $U^{-1}= \code\decode'= U^\dagger$. The following diagrams then commute:
\[ \xymatrix{
S\otimes S \ar[d]_\code \ar[dr]^{\code'}		&		&&								&		S\otimes S 	\\
S \ar[r]_U							& S		&&	S \ar[ur]^{\decode'}\ar[r]_{U^{\dagger}}	& S \ar[u]_\decode 	
}
\]
and $U=\code'\decode$ is the unique unitary satisfying this condition.
\end{enumerate}
\end{proof}

\subsection{The `internal' monoidal tensor of a self-similar structure}\label{internals}
We now demonstrate a close connection between self-similar structures and and untyped (i.e. single-object) categorical properties:
\begin{theorem}\label{standard}
Let $(S,\code,\decode)$ be a self-similar structure of a semi-monoidal category $(\C,\otimes,\tau_{\_,\_,\_})$. Then the code / decode arrows determine a semi-monoidal tensor 
\[ \_ \otimes_\dc \_ : \C(S,S)\times \C(S,S)\rightarrow \C(S,S) \]
on the endomorphism monoid of $S$ given by, for all $a,b\in \C(S,S)$, 
\[ a\otimes_\dc b \ = \ \code (a\otimes b) \decode \in \C(S,S) \]
The associativity isomorphism for this semi-monoidal structure is given by 
\[ \tau_\dc \ = \ \code (\code \otimes 1_S) \tau_{S,S,S} (1_S\otimes \decode) \decode \]
When $(S,\otimes)$ is symmetric, with symmetry isomorphisms $\sigma_{X,Y}\in \C(X\otimes Y,Y\otimes X)$ then $ \_ \otimes_\dc \_ : \C(S,S)\times \C(S,S)\rightarrow \C(S,S)$ is a symmetric semi-monoidal tensor, with symmetry isomorphism $\sigma_\dc\in \C(S,S)$ given by $\sigma_\dc=\code \sigma_{S,S}\decode$.
\end{theorem}
\begin{proof}
This is a standard result of the categorical theory of self-similarity; see \cite{PH98,PH99,PH13} for the general construction, and \cite{PH98,MVL} for numerous examples based on inverse monoids.
\end{proof}

\begin{definition}\label{internal}
Let $(S,\code,\decode)$ be a self-similar structure of a semi-monoidal category $(\C,\otimes,\tau_{\_,\_,\_})$. We refer to the semi-monoidal tensor
\[ \_ \otimes_\dc \_ : \C(S,S)\times \C(S,S)\rightarrow \C(S,S) \]
given in Theorem \ref{standard} above as the {\bf internalisation of $(\ \otimes \ )$ by $(S,\code ,\decode)$}. We similarly refer to the canonical associativity isomorphism $\tau_\dc\in \C(S,S)$ (resp. symmetry isomorphism $\sigma_\dc \in \C(S,S)$ as the {\bf associativity isomorphism (resp. symmetry isomorphism) induced by $(S,\code ,\decode)$}. 
\end{definition}

\begin{remark}\label{notstrict}
It is proved in \cite{PHarxiv} (See also Appendix B of \cite{PH13}) that strict associativity for single-object semi-monoidal categories is equivalent to degeneracy (i.e. the single object being a unit object for the tensor).  Thus, even when $(\C,\otimes)$ is strictly associative,  the associativity isomorphism induced by $(S,\code,\decode)$ given by 
$\tau_\dc = \code (\code \otimes 1_S)  (1_S\otimes \decode) \decode$
is not the identity (at least, provided $S$ is not the unit object for $\_ \otimes_\dc \_ $).
\end{remark}

The following simple corollary of Theorem \ref{standard} above is taken from \cite{PH13}.
\begin{corollary}
Let $(S,\code,\decode)$ be a $\dagger$ self-similar structure of a $\dagger$ semi-monoidal category $(\C,\otimes,\tau_{\_,\_,\_})$. Then 
$ \_ \otimes_\dc \_ : \C(S,S)\times \C(S,S)\rightarrow \C(S,S)$, the internalisation of $\_ \otimes \_$ by $(S,\code,\decode)$, is a $\dagger$ semi-monoidal tensor.
\end{corollary}
\begin{proof}
This is immediate from the property that $\code^\dagger=\decode$, and the definition of $\_ \otimes_\dc \_$ and the canonical isomorphism $\tau_\dc \in \C(S,S)$ in terms of unitaries.
\end{proof}

\section{$\dagger$ Self-similar structures as lax A-H classical structures}\label{CSuptoiso}
We now demonstrate that, up to certain canonical coherence isomorphisms a $\dagger$ self-similar structure $(S,\code,\decode)$ of a symmetric $\dagger$ semi-monoidal category $(\C,\otimes , \tau_{\_,\_,\_},\sigma_{\_,\_} )$ satisfies the axioms for an A-H classical structure.  The precise coherence isomorphisms required are those generated by  
\begin{itemize}
\item The semi-monoidal coherence isomorphisms $\{ \tau_{\_,\_,\_} , \sigma_{\_ ,\_} \}$ of $(\C,\otimes)$
\item The induced coherence isomorphisms $\{ \tau_\dc , \sigma_\dc\}$ of $(\C(S,S),\otimes_{\dc})$ 
\item The semi-monoidal tensors $\_ \otimes\_$ and $\_ \otimes_\dc\_ $
\end{itemize}

\begin{theorem}\label{AHssFull}
Let $(S,\code,\decode)$ be a $\dagger$ self-similar structure of a symmetric $\dagger$ semi-monoidal category $(\C,\otimes , \tau_{\_,\_,\_},\sigma_{\_,\_} )$. Then the following conditions hold:
\begin{itemize}
\item {\bf (Lax associativity)} $\code(\code\otimes1_S)\tau_{S,S,S} = \tau_\dc \code (1_S\otimes \code)$ 
\item {\bf (Lax Frobenius condition)} $\decode \tau_\dc^{-1} \code  =  (1_S\otimes \code) \tau^{-1}_{S,S,S} (\decode\otimes 1_S)$
\item {\bf (Classical structure condition)} $\code\decode=1_S$
\item {\bf (Lax symmetry)} $\sigma_{S,S} \decode = \decode \sigma_\dc$
\end{itemize}
\end{theorem}
\begin{proof} The following proof is based on results of \cite{PH13}. 

Conditions 1. and 2. above follow from the commutativity of the following diagram
\[ 
\xymatrix{
S \ar[r]^\decode \ar[d]_{\tau_\dc} & S \otimes S \ar[rr]^{1_S\otimes \decode} && S\otimes (S\otimes S) \ar[d]^{\tau_{S,S,S}} \\
S						   & S\otimes S \ar[l]^\code				    && (S\otimes S)\otimes S \ar[ll]^{\code \otimes 1_S}
}
\]
which is simply the definition of the induced associativity isomorphism. Condition 3. follows immediately from the definition of a $\dagger$ self-similar structure, and condition 4. is simply the definition of the induced symmetry isomorphism.

\end{proof}

\begin{remark}For the above properties to be taken seriously as lax versions of the axioms for an A-H classical structure, there needs to be some notion of coherence relating the semi-monoidal tensor $\_ \otimes \_ :\C\times \C\rightarrow \C$ and its canonical isomorphisms, to the semi-monoidal tensor $\_ \otimes_\dc \_ : \C(S,S)\times \C(S,S)\rightarrow \C(S,S)$ and its canonical isomorphisms. A general theory of coherence for self-similarity and associativity is given in \cite{PHarxiv}; in Appendix \ref{simplecoherence}, we outline how a simple case of this is also applicable in the $\dagger$ symmetric case.
\end{remark}

It may be wondered whether the induced isomorphisms are necessary in theorem \ref{AHssFull} above -- can we not have a $\dagger$ self-similar structure satisfying analogous conditions solely based on the canonical isomorphisms of $(\C,\otimes)$? The following corollary demonstrates that this can only be the case when $S$ is degenerate  --- i.e. the unit object for some monoidal category.
\begin{corollary}\label{notypedassoc} Let $(S,\code,\decode)$ be a self-similar structure of a semi-monoidal category $(\C,\otimes , \tau_{\_,\_,\_})$. Then the following condition
\begin{itemize}
\item {\bf (Overly restrictive Frobenius condition)} $\decode  \code  =  (1_S\otimes \code) \tau^{-1}_{S,S,S} (\decode\otimes 1_S)$
\end{itemize}
implies that $S$ is degenerate -- i.e. the unit object for some monoidal category.
\end{corollary}
\begin{proof}
By definition, the associativity isomorphism for the internalisation of $(\ \otimes \ )$ is given by 
\[ \tau_\dc = \code (1_S\otimes \code) \tau^{-1}_{S,S,S} (\decode\otimes 1_S) \decode \]
Thus as $\code$ and $\decode$ are mutually inverse unitaries, the overly restrictive Frobenius condition implies that $\tau_\dc=1_S$. However, as proved in \cite{PHarxiv} (see also Appendix B of \cite{PH13}), single-object semi-monoidal categories are strictly associative exactly when their unique object is a unit object of some monoidal category. 
\end{proof}

An alternative perspective of Corollary \ref{notypedassoc} is the following:
\begin{corollary}\label{nodoublecc}
Let $(S,\Delta,\nabla)$ be an A-H classical structure satisfying the precise axioms\footnote{We strongly emphasise that this corollary does not hold if we allow the axioms of Definition \ref{AHFA} to hold up to canonical isomorphism, as demonstrated in Theorems \ref{AHssFull} and \ref{AHssStrict}.}  of Definition \ref{AHFA}. Unitarity of $\Delta$ implies that $S$ is the unit object of a monoidal category.
\end{corollary}

Despite Corollaries \ref{notypedassoc} and \ref{nodoublecc} above, it is certainly possible for a self-similar structure to satisfy all the axioms for a Frobenius algebra up to a single associativity isomorphism; however, this must be the induced associativity isomorphism of Definition \ref{internal}, as we now demonstrate:

\begin{theorem}\label{AHssStrict}
Let $(S,\code,\decode)$ be a $\dagger$ self-similar structure of a strictly associative $\dagger$ semi-monoidal category $(\C,\otimes , (\ )^\dagger)$. Then the defining conditions of an A.-H. $\dagger$ Frobenius algebra are satisfied up to a single associativity isomorphism as follows:
\begin{itemize}
\item $ \code (\code \otimes 1_S) = \tau_\dc \code (1_s \otimes \code)$
\item $(\decode\otimes 1_S)(1_S\otimes \decode) = \decode \tau_\dc^{-1} \code$
\end{itemize}
\end{theorem}
\begin{proof}
This is simply the result of Theorem \ref{AHssFull} in the special case where the monoidal tensor $\_ \otimes \_ :\C\times \C\rightarrow \C$ is strictly associative. Note that even though $\_ \otimes \_ :\C\times \C \rightarrow \C$ is  strictly associative, 
its internalisation $\otimes_\dc:\C(S,S)\times \C(S,S)\rightarrow \C(S,S)$ cannot be strictly associative; rather, from Theorem \ref{standard}  the required associativity isomorphism is given by $\tau_\dc= \code (\code\otimes 1_S)(1_S\otimes \decode)\decode \neq 1_S$.
\end{proof}


\section{An illustrative example}
In the following sections, we will present the theory behind an example of a $\dagger$ self-similar structure that determines matrix representations of arrows in a similar manner to how classical structures in finite-dimensional Hilbert space determine matrix representations of linear maps. Our example is deliberately chosen to be as `non-quantum' as possible, in order to explore the limits of the interpretations of pure category theory: it comes from a setting where all isomorphisms are unitary, all idempotents commute, and the lattice of idempotents satisfies distributivity rather than some orthomodularity condition.
Many of these properties are determined by the particular form of dagger operation used; we will work with {\em inverse categories}.  

\section{Inverse categories as $\dagger$ categories}
Inverse categories arose from the algebraic theory of semigroups, but the extension to a categorical definition is straightforward and well-established. We also refer to \cite{RC} for the more general {\em restriction categories} that generalise inverse categories in the same way that restriction monoids generalise inverse monoids.
\begin{definition}\label{inv}{\em Inverse categories}\\
An {\bf inverse category} is a category $\C$ where every arrow $f\in \C(X,Y)$ has a unique {\bf generalised inverse} $f^\ddag\in \C(Y,X)$ satisfying $ff^\ddag f=f$ and $f^\ddag ff^\ddag =f^\ddag$.
A single-object inverse category is called an {\bf inverse monoid}. 
\end{definition}

\begin{remark}{\bf Uniqueness of generalised inverse operations}
Inverse monoids and semigroups were defined and studied long before inverse categories; the definition of an inverse category is thus rather `algebraic'  in nature, given by requiring the existence of unique arrows satisfying certain properties -- this is in contrast to a more functorial definition. 
However, uniqueness implies that there can be at most one (object-indexed)  operation $(\ )_{XY} : \C(X,Y)\rightarrow \C(Y,X)$ that takes each arrow to some generalised inverse satisfying the above axioms.  We will therefore treat $(\ )^\ddag$ as an (indexed) bijection of hom-sets, and ultimately (as we demonstrate in Theorem \ref{inversedagger} below) a contravariant functor.
\end{remark}

The following result is standard, and relates generalised inverses and idempotent structures of inverse monoids (see, for example \cite{MVL}). 
\begin{lemma}\label{invbasics}
Let $M, (\ )^\ddag$ be an inverse monoid. Then for all $a\in M$, the element $a^\ddag a$ is idempotent, and the set of idempotents $E_M$ of $M$ is a commutative submonoid of $M$ where every element is its own generalised inverse.
\end{lemma}
\begin{proof}
These are standard results of inverse semigroup theory, relying heavily on the {\em uniqueness} of generalised inverses.  The key equivalence between commutativity of idempotents and uniqueness of generalised inverses is due to \cite{MP}.
\end{proof}

Based on the above, the following is folklore:
\begin{theorem}
Let $\C,(\ )^\ddag$ be an inverse category. Then the operation $(\ )^\ddag$ is a dagger operation, and all isomorphisms of $\C$ are unitary.
\end{theorem}\label{inversedagger}
\begin{proof}{\em The technical results we require are straightforward generalisations of well-established inverse semigroup theory, so are simply given in outline.} \\
First observe that it is implicit from the definition that, on objects $X^\ddag=X\in Ob(\C)$.  
We now prove that $(\ )^\ddag$, with this straightforward extension to objects, is a contravariant involution. 

To demonstrate contravariant functoriality, observe that $(gf) (gf)^\ddag gf = gf$ for all $f\in \C(X,Y)$ and $g\in \C(Y,Z)$. 
However, $ff^{\ddag}$ and $g^{\ddag}g$ are both idempotents of $Y$, and thus commute. Hence $(gf)f^{\ddag}g^{\ddag} (gf) = gg^{\ddag}g ff^\ddag f = gf$ and so $(gf)^\ddag=f^\ddag g^\ddag$ as required. 

To see that $(\ )^\ddag$ is involutive, note that by definition $f^{\ddag}\left( f^{\ddag} \right)^\ddag f^\ddag = f^\ddag$, for all $f\in \C(X,Y)$. However, also from the definition, $f^\ddag f f^\ddag =f^\ddag$ and again by uniqueness, $\left(f^\ddag\right)^\ddag = f$. 

Thus $(\ )^\ddag$ is a contravariant involution that acts trivially on objects. To see that all isomorphisms are unitary, consider an arbitrary isomorphism $u\in \C(X,Y)$. Then trivially, $uu^{-1}u=u\in \C(X,Y)$ and $u^{-1}uu^{-1}=u^{-1}\in \C(Y,X)$. Uniqueness of generalised inverses then implies that $u^{-1}=u^\ddag$, and hence $u$ is unitary.
\end{proof}

\begin{corollary}
Let $\C$ be an inverse category with a semi-monoidal tensor $(\ \otimes \ )$. Then $(\C,\otimes ,(\ )^\ddag)$ is a dagger semi-monoidal category.
\end{corollary}
\begin{proof} Given arbitrary $f\in \C(A,B)$ and $g\in \C(X,Y)$, then by functoriality 
\[ (f\otimes g)(f^\ddag \otimes g^\ddag) (f\otimes g) = (ff^\ddag f \otimes gg^\ddag g) = (f\otimes g) \]
However, by definition
\[ (f\otimes g)(f \otimes g)^\ddag (f\otimes g) = (ff^\ddag f \otimes gg^\ddag g) = (f\otimes g) \]
and by uniqueness, $(f\otimes g)^\ddag= f^\ddag \otimes g^\ddag$.
Also, since all isomorphisms are unitary, all canonical isomorphisms are unitary.
\end{proof}

\subsection{The natural partial order on hom-sets}
All inverse categories have a naturally defined partial order on their hom-sets:
\begin{definition} Let $\mathcal C, (\ )^{\ddag}$ be an inverse category. For all $A,B\in Ob({\mathcal C})$, the relation $\unlhd_{A,B}$ is defined on ${\mathcal C}(A,B)$, as follows:
\[ f\unlhd_{A,B} g \ \mbox{ iff } \ \exists \ e^2=e\in {\mathcal C}(A,A) \ s.t. \ f=ge \]
It is immediate that, for all $A,B\in Ob({\mathcal C})$, the relation $\unlhd_{A,B}$ is a partial order on ${\mathcal C}(A,B)$, called the {\bf natural partial order}.\\

\noindent{\bf Convention:} When it is clear from the context, we omit the subscript on $\unlhd$.
\end{definition}

We may rewrite the above non-constructive definition more concretely:
\begin{lemma}
Given $f\unlhd g\in {\mathcal C}(X,Y)$, in some inverse category, then $f=gf^{\ddag}f$.
\end{lemma}
\begin{proof}
By definition, $f=ge$, for some $e^{2}=e\in {\mathcal C}(X,X)$. Thus $fe=f$, since $e$ is idempotent. From the defining equation for generalised inverses, 
$f=ff^{\ddag}f=gef^{\ddag}f$. As $f^{\ddag}f$ is idempotent, and idempotents commute, $f=gf^{\ddag}fe$. However, we have already seen that $fe=f$, and hence $f=gf^{\ddag}f$.\\

\end{proof} 

A very useful tool in dealing with the natural partial order is the following lemma, which is again a classic result of inverse semigroup theory rewritten in a categorical setting (see also \cite{GOI3} where it is rediscovered under the name `passing a message through a channel').

\begin{lemma}\label{passingidempotents} {\bf Pushing an idempotent through an arrow} Let $\C,(\ )^\ddag)$ be an inverse category. Then for all $f\in \C(X,Y)$, and $e^2=e\in \C(X,X)$, there exists an idempotent $e'^2=e'\in \C(Y,Y)$ satisfying $e'f=fe$.
\end{lemma}
\begin{proof} We define $e'=fef^\ddag\in \C(Y,Y)$.  By Lemma \ref{invbasics}, $e'^2=fef^\ddag fef^\ddag = ff^\ddag feef^\ddag$ $= fef^\ddag=e'$. 
Further, $e'f=fef^\ddag f = ff^\ddag f e=fe$, as required. 
\end{proof}

\begin{proposition}\label{cong} The natural partial order is a congruence --- that is, given $f\unlhd h \in {\mathcal C}(X,Y)$ and $g\unlhd k\in {\mathcal C}(Y,Z)$, then 
$gf\unlhd kh \in {\mathcal C}(X,Z)$.
\end{proposition}
\begin{proof}
By definition, there exists idempotents $p^2=p\in {\mathcal C}(X,X)$ and $q^2=q\in {\mathcal C}(Y,Y)$ such that $f=hp\in {\mathcal C}(X,Y)$ and $g=kq\in {\mathcal C}(Y,Z)$, 
and hence $gf=hpkq\in {\mathcal C}(X,Z)$. We now use the `passing an idempotent through an arrow' technique of Lemma \ref{passingidempotents} to deduce the existence of an idempotent $p'\in {\mathcal C}(X,X)$ such that $pk=kp'\in {\mathcal C}(X,Y)$. Hence $gf=khp'q$. However, by Part 3. of Lemma \ref{passingidempotents}, $p'q$ is idempotent, and hence $gf \unlhd kh$, as required.\end{proof}

\begin{corollary}
Every locally small inverse category $(\C ,(\ )^\ddag)$ is enriched over the category $\bf Poset$ of partially ordered sets.
\end{corollary}
\begin{proof}
Expanding out the definition of categorical enrichment will demonstrate that the crucial condition is that proved in Proposition \ref{cong} above.
\end{proof}

\subsection{A representation theorem for inverse categories}
A classic result of inverse semigroup theory is the Wagner-Preston representation theorem  \cite{GP,VW} which states that every inverse semigroup $S$ is isomorphic to some semigroup of partial isomorphisms on some set. This implies the usual representation theorem for groups as subgroups of isomorphisms on sets. There exists a natural generalisation of this theorem to inverse categories:

\begin{definition}
The inverse category $\bf pIso$ 
is defined as follows:
\begin{itemize}
\item {\bf (Objects)} All sets.
\item {\bf (Arrows)} ${\bf pIso}(X,Y)$ is the set of all partial isomorphisms from $X$ to $Y$. In terms of diagonal representations, it is the set of all subssets $f\subseteq Y\times X$  satisfying
\[ b=y\ \Leftrightarrow \ y=a \ \ \forall \ (b,a),(y,x)\in f \]
\item {\bf (Composition)} This is inherited from the category $\bf Rel$ of relations on sets in the obvious way.
\item {\bf (Generalised inverse)} This is given by $f^\ddag = \{ (x,y) : (y,x)\in f$; the obvious restriction of the relational converse. 
\end{itemize}
The category $\pIso$ has zero arrows, given by $0_{XY}=\emptyset \subseteq Y\times X$.  This is commonly used to define a notion of {\bf orthogonality} by 
\[ f\perp g\in \C(X,Y) \ \ \Leftrightarrow \ \ g^\ddag f=0_X \mbox{ and } gf^\ddag=0_Y \]
\end{definition}

\begin{remark}The category $(\PI,\uplus)$ is well-equipped with self-similar structures; one of the most heavily-studied \cite{PH98,MVL,PH99} is the natural numbers $\mathbb N$, although any countably infinite set will suffice. As demonstrated in an Appendix to \cite{PH13a}, there is a 1:1 correspondence between self-similar structures at $\mathbb N$ and points of the Cantor set (excluding a subset of measure zero).  Other examples include the Cantor set itself \cite{PH98,PH99} and other fractals \cite{blatant}.\end{remark}

\begin{theorem}
Every locally small inverse category $\left( \C,(\ )^\ddag \right)$  is isomorphic to some subcategory of $({\bf pIso}, (\ )^\ddag)$.
\end{theorem}
\begin{proof}
This is proved in \cite{CH}, and significantly prefigured (for small categories)  in \cite{RC}.
\end{proof}

The idempotent structure and natural partial ordering on $\pIso$ is particularly well-behaved, as the following standard results demonstrate:

\begin{proposition}
\begin{enumerate} 
\item The natural partial order of $\pIso$ may be characterised in terms of diagonal representations by $f\unlhd g\in \pIso(X,Y)$ iff $f\subseteq g \in Y\times X$.
\item All idempotents $e^2=e\in \pIso(X,X)$ are simply partial identities $1_{X'}$ for some $X'\subseteq X$, and thus $\pIso$ is isomorphic to its own Karoubi envelope.
\item The meet and join w.r.t. the natural partial order are given by, for all $f,g\in \pIso(X,Y)$  
\[ f\vee g = f\cup g \ \mbox{ and } f\wedge g = f\cap g \]
when these exist. Therefore, set of idempotents at an object is a distributive lattice. 
\item Given an arbitrarily indexed set $\{ f_j \in \pIso(X,Y) \}_{j\in J}$ of pairwise-orthogonal elements, together with arbitrary $a\in \PI(W,X)$ and $b\in \PI(Y,Z)$,  
 then $\bigvee_{j\in J} f_j \in \pIso(X,Y)$ exists, as does $\bigvee_{j\in J}b f_ja \in \pIso(W,Z)$, and
\[ b\left( \bigvee_{j\in J} f_j\right) a \ = \  \bigvee_{j\in J} \left( bf_ja \right) \]
\end{enumerate}
\end{proposition}
\begin{proof}
These are all standard results for the theory of inverse categories; 
1. is a straightforward consequence of the definition of the natural partial order, and 2.-4. follow as simple corollaries.
\end{proof}

\section{Monoidal tensors and self-similarity in $\pIso$}
We have seen that $\pIso$ is a $\dagger$ category;  it is also a $\dagger$ monoidal category with respect to two distinct monoidal tensors - the Cartesian product $\_ \times \_$ and the disjoint union $\_ \uplus \_$. For the purposes of this paper, we will study the disjoint union. We make the following formal definition:

\begin{definition} We define the {\bf disjoint union} $\uplus : \PI\times \PI \rightarrow \PI$ to be the following monoidal tensor:
\begin{itemize}
\item {\bf (Objects)} $A\uplus B = A\times \{0\} \cup B\times \{ 1\}$, for all $A,B\in Ob(\PI)$.
\item {\bf Arrows)} Given $f\in \PI(A,B)$ and $g\in \PI(X,Y)$, we define $f\uplus g =inc_{00}( f)  \cup inc_{11}(g)\subseteq (B\uplus Y) \times (A\uplus X)$ where $inc_{00}$ is the canonical (for the Cartesian product) isomorphism $B\times A\cong B\times \{0\} \times A\times \{0\}$, and similarly, $inc_{11}:Y\times X \cong Y\times \{ 1\} \times X\times \{1\}$.
\end{itemize}
It is immediate that $(\PI,\uplus)$ is a $\dagger$-monoidal tensor since, as a simple consequence of the definition of generalised inverses, all isomorphisms are unitary.
\end{definition}

By contrast with the behaviour of disjoint union in (for example) the category of relations, it is neither a product nor a coproduct on $\PI$. Despite this, it has analogues of projection \& inclusion maps:

\begin{definition}\label{projinc}
Given $X,Y\in Ob(\PI)$, the arrows $\iota_l\in \PI(X,X\uplus Y)$ and $\iota_r\in \PI(Y,X\uplus Y)$ are defined by 
$\iota_l(x)= (x,0)\in X\uplus Y$ and $\iota_r(y)=(y,1)\in X\uplus Y$.
By convention, we denote their generalised inverses by $\pi_l:\in\PI(X\uplus Y,X)$ and $\pi_r\in\PI (X\uplus Y,Y)$, giving
\[ \xymatrix{
						&	X \uplus Y \ar@/^8pt/[dl]^{\pi_l} \ar@/_8pt/[dr]_{\pi_l} 	&			\\
X \ar@/^8pt/[ur]^{\iota_l}		&												& Y \ar@/_8pt/[ul]_{\iota_r}
}
\]
Following \cite{PH98} we refer to these arrows as the {\bf projection} and {\bf inclusion} arrows; they are sometimes \cite{HA,AHS} called {\bf quasi-projections / injections}, in order to emphasise that they are not derived from categorical products / coproducts.
By construction, the projections / inclusions satisfy the following four identities:
\[ \pi_r\iota_l=0_{XY} \ \ , \ \ \pi_l\iota_r = 0_{YX} \  \ , \ \ \pi_l\iota_l=1_X \  \ , \ \ \pi_r\iota_r = 1_Y \]
\end{definition}

As noted in \cite{PH98,PH99}, the above arrows can be `internalised' by a self-similar structure $(S,\code,\decode)$, in a similar way to canonical isomorphisms (see Section \ref{internals}). Doing so will give an embedding of a well-studied inverse monoid into $\PI(S,S)$.

\begin{definition}{\em Polycyclic monoids}\label{p2}\\
The $2$ generator {\bf polycyclic monoid} $P_2$ is defined in \cite{NP} to be the inverse monoid given by the generating set $\{ p,q \}$, together with the relations
\[ pp^{-1}=1=qq^{-1} \ \ ,\ \ pq^{-1}=0 =qp^{-1} \]
\end{definition}

\begin{remark} This inverse monoid is also familiar to logicians as the (multiplicative part of) the {\em dynamical algebra} of \cite{DR,GOI}. It is also familiar from the theory of state machines as the syntactic monoid of a pushdown automaton with a binary stack \cite{RG}, and to pure mathematicians as the monoid of partial homeomorphisms of the Cantor set \cite{PH98}.
\end{remark}

The following result on polycyclic monoids will prove useful:
\begin{lemma}\label{polyprop}
 $P_2$ is {\em congruence-free}; i.e. the only composition-preserving equivalence relations on $P_2$ are either the universal congruence $r\sim s$ for all $r,s\in P_2$, or the identity congruence $r\sim s \ \Leftrightarrow \ r=s$ for all $r,s\in P_2$.
\end{lemma}
\begin{proof}
This is a special case of a general result of \cite{NP}. 
Congruence-freeness is an example of Hilbert-Post completeness; categorically, it is closely related to the `no simultaneous strictification' theorem of \cite{PHarxiv}.
\end{proof}

The following result, generalising a preliminary result of  \cite{PH98,PH99}, makes the connection between embeddings of polycyclic monoids and internalisations of projection/ injection arrows of $\PI$ precise:
\begin{theorem}\label{SSP2}
Let $S$ be a self-similar object (and hence a $\dagger$ self-similar object) of $\PI$. We say that an inverse monoid homomorphism $\phi:P_2\rightarrow \PI(S,S)$ is a {\bf strong embedding} when it satisfies the condition 
\[ \phi(p^\dagger p) \vee \phi(q^\dagger q) \ = \ 1_S \]
Then every strong embedding $\phi :P_2\rightarrow \PI(S,S)$ uniquely determines, and is uniquely determined by, a $\dagger$ self-similar structure at $S$.
\end{theorem}
\begin{proof}
Let $\pi_l,\pi_r\in\PI(S\uplus S,S$ and $\iota_l,\iota_r\in \PI(S,S\uplus S)$ be the projections / inclusions of Definition \ref{projinc}, and let $(S,\code,\decode)$ be a self-similar structure. We define $\phi_\dc:P_2\rightarrow \PI(S,S)$ by its action on the generators of $P_2$, giving 
\[ \phi_\dc(p)=\pi_l\decode \ \mbox{ and } \ \phi_\dc(q) = \pi_r\decode  
\] 
Their generalised inverses are then $\phi_\dc(p^\ddag) = \code\iota_l$ and $\phi_\dc (q^\ddag) = \code \iota_r$. Thus 
\[  \phi_\dc(p)\phi_\dc(p^\ddag) = \pi_l\decode \code\iota_l = 1_S = \pi_r\decode \code \iota_r = \phi_\dc(q)\phi_\dc (q^\ddag) \]
Similarly, $ \phi_\dc(p)\phi_\dc (q^\ddag) = \pi_l\decode  \code \iota_r = 0_{S} =  \pi_r\decode\code\iota_l = \phi_\dc(q)\phi_\dc (p^\ddag)$ and so $\phi_\dc$ is a homomorphism. Since $P_2$ is congruence-free it is also an embedding. To demonstrate that it is also a strong embedding, 
\[ 1_S= \code 1_{S\uplus S} \decode = \code ( \iota_l\pi_l \vee \iota_r\pi_r) \decode \]
\[ = \code \iota_l\pi_l\decode \vee \code \iota_r\pi_r \decode = \phi_\dc(p^\ddag p) \vee \phi_\dc (q^\ddag q) \]
as required. Further, given another self-similar structure $(S,c,d)$ satisfying $\phi_{dc}=\phi_\dc$, then $\code=c\in \PI(S\uplus S,S)$ and $\decode=d\in \PI(S,S\uplus S)$.

Conversely, let  $\phi:P_2\rightarrow \PI(S,S)$ be  a strong embedding, and consider the two arrows
$\iota_l\phi(p)\in \PI(S,S\uplus S)$ and $\iota_r \phi(q) \in \PI(S,S\uplus S)$.
it is straightforward that these are orthogonal; we thus define
\[ \decode_\phi =  \iota_l\phi(p)\vee  \iota_r \phi(q)  \in \PI(S,S\uplus S) \]
and take $\code_\phi=\decode_\dc^\ddag$.
The strong embedding condition implies that $\code_\phi \decode_\phi = 1_S$ and $\decode_\phi \code_\phi = 1_{S\uplus S}$; thus we have a self-similar structure, as required. 
 Further, given another strong embedding $\psi:P_2\rightarrow \PI(S,S)$, then $\code_\phi=\code_\psi$ iff $\phi = \psi$.
 \end{proof}

\subsection{Matrix representations from self-similar structures}
We are now in a position to demonstrate how self-similar structures in $\PI$ determine matrix representations of arrows.
\begin{theorem}
Let $S\in Ob(\PI)$ be a self-similar object. Then every self-similar structure $(S,\code,\decode)$ determines matrix representations of arrows of $\PI(S,S)$.
\end{theorem}
\begin{proof} We use the correspondence between self-similar structures and strong embeddings of polycyclic monoids given in Theorem \ref{SSP2}. Given arbitrary $f\in \PI(S,S)$, we define $[f]_\dc$, the {\bf matrix representation of $f$ determined by $(S,\code,\decode)$} to be the following matrix:
\[ [ f ]_\dc = \left( \begin{array}{ccc} 	\phi_\dc(p) f \phi_\dc(p^\ddag)  	& \ \ 	&  \phi_\dc(p)f \phi_\dc(q^{\ddag}) \\ 
														&	&							\\
						   	\phi_\dc(q) f \phi_\dc(p^{\ddag})  	& \ \ 	& \phi_\dc(q)f \phi_\dc(q^{\ddag}) \\ 	
							\end{array} \right) \]
Given two such matrices of this form, we interpret their {\bf matrix composition} as follows:
\[  \left( \begin{array}{cc} g_{00} & g_{01} \\ g_{10} & g_{11} \end{array}\right)  \left( \begin{array}{cc} f_{00} & f_{01} \\ f_{10} & f_{11} \end{array} \right)  = \left( \begin{array}{ccc}
g_{00} f_{00} \vee g_{01}f_{10} &	\ \ 	& g_{00}f_{01} \vee g_{01}  f_{11} \\ 
						&		&							\\
g_{10}f_{00} \vee g_{11}f_{10} & 	\ \ 	& g_{10}f_{01} \vee g_{11}f_{11} \end{array}\right) 
\]
that is, the usual formula for matrix composition, with summation interpreted by join in the natural partial order -- provided that the required joins exist. We  prove that this composition is defined for matrix representations determined by a fixed self-similar structure.

{\em In what follows, we abuse notation, for clarity, and refer to $p,q,p^\ddag ,q^\ddag \in \PI (S,S)$ instead of $\phi_\dc( p),\phi_\dc(q),\phi_\dc(p^\ddag),\phi_\dc(q^\ddag) \in \PI(S,S)$. As this proof is based on a single fixed self-similar structure at $S$, we may do this without ambiguity}.

Consider the entry in the top left hand corner of $[g]_\dc[f]_\dc$. Expanding out the definition will give this as $pgp^\ddag pfp^\ddag \vee pgq^\ddag q f p^\ddag$.
To demonstrate that these two terms are orthogonal, $\left(pgp^\ddag pfp^\ddag \right)^\ddag \left( pgq^\ddag q f p^\ddag \right) = pf^\ddag p^\ddag p g^\ddag p^\ddag p gq^\ddag qfp^\ddag$.
Appealing to the `pushing an idempotent through an arrow' technique of Proposition \ref{passingidempotents} gives the existence of some idempotent $e^2=e$ such that 
\[ \left(pgp^\ddag pfp^\ddag \right)^\ddag \left( pgq^\ddag q f p^\ddag \right) = pf^\ddag p^\ddag p  e g^\ddag  gq^\ddag qfp^\ddag \]
Again appealing to this technique gives the existence of some idempotent $E^2=E$ such that $ \left(pgp^\ddag pfp^\ddag \right)^\ddag \left( pgq^\ddag q f p^\ddag \right) = pf^\ddag p^\ddag E  p q^\ddag qfp^\ddag$. 
However, $pq^\ddag=0$ and hence$\left(pgp^\ddag pfp^\ddag \right)^\ddag \left( pgq^\ddag q f p^\ddag \right) = 0$.
as required.
An almost identical calculation will give that  $\left( pgq^\ddag q f p^\ddag \right)^\ddag \left( pgp^\ddag pfp^\ddag \right) = 0$
 and thus these two terms are orthogonal, so the required join exists.

The proof of orthogonality for the other three matrix entries is almost identical; alternatively, it may be derived using  the obvious isomorphism of $P_2$ that interchanges the roles of $p$ and $q$.

It remains to show that composition of matrix repesentations of elements coincides with composition of these elements; we now prove that $[g]_\dc [f]_\dc = [gf]_\dc$. By definition,
\[ [gf]_\dc = \left(\begin{array}{ccc}
							pgfp^\ddag & & pgf q^\ddag \\
									& & 		\\
							qgfp^\ddag & & q gf q^\ddag 
							\end{array}\right) 
\]
As the (implicit) embedding of $P_2$ is strong, $1_S=p^\ddag p \vee q^\ddag q$. We may then 
substitute $g(p^\ddag p \vee q^\ddag q)f$ for $gf$ in the above to get 
\[ [gf]_\dc = \left(\begin{array}{ccc}
							pg(p^\ddag p \vee q^\ddag q)fp^\ddag & & pg(p^\ddag p \vee q^\ddag q)f q^\ddag \\
									& & 		\\
							qg(p^\ddag p \vee q^\ddag q)fp^\ddag & & q g(p^\ddag p \vee q^\ddag q)f q^\ddag 
							\end{array}\right) 
\]
Expanding this out using the distributivity of composition over joins gives the definition of $[g]_\dc [f]_\dc$, and hence $[gf]_\dc=[g]_\dc[f]_\dc$, as required. 

Finally, we need to prove that the representation of arrows as matrices determined by the self-similar structure $(S,\code,\decode)$ is faithful --- that is, $a=b\in \PI(S,S)$ iff $[b]_\dc = [a]_\dc$ (where equality of matrices is taken as component-wise equality). 

The $(\Rightarrow)$ implication is immediate from the definition. For the other direction, $[b]_\dc=[a]_\dc$ when the following four identities are satisfied:
\[ 
 \begin{array}{ccc}
							pap^\ddag = pbp^\ddag & & pa q^\ddag = pbq^\ddag\\
									& & 		\\
							qap^\ddag  = qbp^\ddag & & q a q^\ddag =qbq^\ddag
							\end{array}
\]
Prefixing/ suffixing each of these identities with the appropriate choice selection taken from $\{p,q,p^\ddag,q^\ddag \}$ will give the following identities:
\[ 
 \begin{array}{ccc}
							p^\ddag pap^\ddag p = p^\ddag pbp^\ddag p & \ \ \ \ & p^\ddag pa q^\ddag  q=  p^\ddag pbq^\ddag q \\
									& & 		\\
							 q^\ddag qap^\ddag p = q^\ddag  qbp^\ddag  p & \ \ \ \ &  q^\ddag q a q^\ddag  q = q^\ddag qbq^\ddag q
							\end{array}
\]
Now observe that these four elements are pairwise-orthogonal. We may take their join, and appeal to distributivity of composition over join to get
\[ (p^\ddag p \vee q^\ddag q ) a (p^\ddag p \vee q^\ddag q) =  (p^\ddag p \vee q^\ddag q ) a (p^\ddag p \vee q^\ddag q) \]
However, as the implicit embedding of $P_2$ is strong, $ (p^\ddag p \vee q^\ddag q )=1_S$ and thus $a=b$, as required.

\end{proof}

\begin{remark}It may seem somewhat disappointing that a self-similar structure $(S,\code,\decode)$ simply determines $(2\times 2)$ matrix representations of arrows of $\PI(S,S)$, rather than matrix representations of arbitrary orders. This is not quite the case, but there is a subtlety to do with the behaviour of the internalisation of the tensor $\_ \uplus \_ : \PI\times \PI\rightarrow \PI$. It is immediate from the definition that the internalisation of this tensor by a self-similar structure has the obvious matrix representation:
$[f \uplus_\dc g]_\dc \ = \ \left(\begin{array}{cc} f & 0_S \\ 0_S & g \end{array}\right)$. 
However, recall from Remark \ref{notstrict} that the internalisation $\_ \otimes_\dc \_$  of an arbitrary tensor $\_ \otimes \_ $ can never be strictly associative, even when $\_ \otimes \_$ itself is associative. Thus, in our example in $(\PI,\uplus)$, arbitrary $(n\times n)$ matrices, in the absence of additional bracketing information, cannot {\em ambiguously} represent arrows.  It is of course possible to have unambiguous $n\times n$ matrix representations that are determined by binary treeS whose leaves are labelled with a single formal symbol, and whose nodes are labelled by self-similar structures at $S$ -- however, this is beyond the scope of this paper!
\end{remark}

\subsection{Isomorphisms of self-similar structures as `changes of matrix representation'}
We have seen in Proposition \ref{unique} that $\dagger$ self-similar structures are unique up to unique unitary. We now relate this to the correspondence  in $(\PI,\uplus)$ between $\dagger$ self-similar structures, strong embeddings of $P_2$, and matrix representations.

\begin{lemma}
Let $(S,\code ,\decode)$ and $(S,c,d)$ be two $\dagger$ self-similar structures at the same object of $(\PI,\uplus)$, and let $U$ be the unique isomorphism (following Proposition \ref{unique}) making the following diagram commute: 
\[ \xymatrix{
S\otimes S \ar[d]_\code \ar[dr]^{\code'}		&		&&								&		S\otimes S 	\\
S \ar[r]_U							& S		&&	S \ar[ur]^{\decode'}\ar[r]_{U^{\dagger}}	& S \ar[u]_\decode 	
}
\]
The two strong embeddings $\phi_\dc,\phi_{(c,d)} : P_2\rightarrow \PI(S,S)$ determined by these self-similar structures (as in Theorem \ref{SSP2}) are mutually determined by the following identities:
\[ \begin{array}{ccc}
\phi_{(c,d)}(p) = \phi_\dc (p) U^{-1} & \ \ \ \ & \phi_{(c,d)}(q) = \phi_\dc (q)U^{-1} \\
							&				&					\\
\phi_{(c,d)}(p^\ddag) = U \phi_\dc (p^\ddag)  & \ \ \ \ & \phi_{(c,d)}(q^\ddag) = U\phi_\dc (q^\ddag) \\
\end{array}
\] 
\end{lemma}
\begin{proof} 
By construction, $c=U\code$ and $d=\decode U^{-1}$. Thus $\phi_{(c,d)}(p) = \pi_ld = \pi_l\decode U^{-1} = \phi_{\dc}(p)U^{-1}$. Taking duals (generalised inverses) gives  $\phi_{(c,d)}(p^\ddag)=U\code\iota_l = U\phi_\dc(p^\ddag)$. The other two identities follow similarly.
\end{proof}
The above connection between the embeddings of $P_2$ given by two self-similar structures allows us to give the transformation between matrix representations of arrows given by two self-similar structures:
\begin{theorem}\label{changebasis}
Let $(S,\code ,\decode)$ and $(S,c,d)$ be two self-similar structures at the same object of $(\PI,\uplus)$, and let the matrix representations of some arrow $f\in \PI(S,S)$ given by $(S,\code,\decode)$ and $(S,c,d)$ respectively be 
\[ [f]_\dc = \left(\begin{array}{cc} \alpha & \beta \\ \gamma & \delta \end{array} \right) \ \ \mbox { and } \ \  [f]_{(c,d)} = \left(\begin{array}{cc} \alpha' & \beta' \\ \gamma' & \delta' \end{array} \right) \]
Then $[f]_{(c,d)}$ is given in terms of $[f]_\dc$ by the following matrix composition:
\[  \left( \begin{array}{ccc} \alpha' &	\ &  \beta' \\ 	& 	& \\ \gamma' &	\ &  \delta' \end{array}\right)  = 
\left(\begin{array}{ccc} u_{00}^\ddag & \ & u_{10}^\ddag \\	& & 	\\  u_{01}^\ddag & \	&  u_{11}^\ddag \end{array}\right) 
 \left( \begin{array}{ccc} \alpha &	\ &  \beta \\ 	& & \\ 	\gamma &	\ &  \delta \end{array}\right)
\left(\begin{array}{ccc} u_{00} &	\ &  u_{01} \\ 	& & 	\\ u_{10} &	 \ &  u_{11} \end{array}\right) 
\]
where 
\[ 
\left(\begin{array}{ccc} u_{00} & \ \  &  u_{01} \\ 	& & 	\\ u_{10} & \ \ & u_{11} \end{array}\right) = 
\left( \begin{array}{ccc} \phi_{\dc}(p)\phi_{(c,d)}(p^\ddag) & 	\ \ 	& \phi_{\dc}(p)\phi_{(c,d)}(q^\ddag) \\ 
					& & 			\\
\phi_\dc(q)\phi_{(c,d)}(p^\ddag) & 	\ \ 	& \phi_\dc(q)\phi_{(c,d)}(q^\ddag) \end{array} \right) \]
\end{theorem}
\begin{proof}
Long direct calculation, expanding out the definition of the above matrix representations, will demonstrate that 
\[ \left( \begin{array}{ccc} \alpha' &\ \  & \beta'  \\  & \ \ & \\ \gamma' & & \delta' \end{array} \right) = 
 \left( \begin{array}{ccc} 	\phi_{(c,d)}(p) f \phi_{(c,d)}(p^\ddag)  	& \ \ 	&  \phi_{(c,d)}(p)f \phi_{(c,d)}(q^{\ddag}) \\ 
														&	&							\\
						   	\phi_{(c,d)}(q) f \phi_{(c,d)}(p^{\ddag})  	& \ \ 	& \phi_{(c,d)}(q)f \phi_{(c,d)}(q^{\ddag}) \\ 	
							\end{array} \right) \]
as a consequence of the identities 
\[ \phi_\dc(p^\ddag p) \vee \phi_\dc(q^\ddag q) = 1_S = \phi_{(c,d)}(p^\ddag p ) \vee \phi_{(c,d)}(q^\ddag q) \]
\end{proof}

\subsection{Diagonalisations  of matrices via isomorphisms of self-similar structures}
A useful application of basis changes in linear algebra is to construct diagonalisations of matrices. For a matrix $M = \left( \begin{array}{cc} A & B \\ C & D \end{array}\right) $ over a vector space $V=V_1\oplus V_2$, a {\em diagonalisation} is a linear isomorphism $D$ satisfying 
$D^{-1} M D = \left(\begin{array}{cc} A' & 0 \\ 0 & B' \end{array}  \right)$,
for some elements $A',B'$. We demonstrate how this notion of diagonalisation has a direct analogue at self-similar objects of $(\PI,\uplus)$, and provide a necessary and sufficient condition (and related construction) for an arrow to be diagonalised by an isomorphism of self-similar structures.

\begin{definition}{\em Diagonalisation at self-similar objects of $(\PI,\uplus)$}\\
Let $(S,\code,\decode)$ be a self-similar structure of $(\PI,\uplus)$ and let $\in \PI(S,S)$ be an arrow with matrix representation $[f]_\dc = \left( \begin{array}{cc} \alpha & \beta \\ \gamma & \delta \end{array}\right)$. We define a {\bf diagonalisation} of this matrix representation to be a self-similar structure $(S,c,d)$ such that $[f]_{(c,d)} = \left(\begin{array}{cc} \lambda & 0 \\ 0 & \mu \end{array} \right)$, so the matrix conjugation given in Theorem \ref{changebasis} satisfies
\[  \left( \begin{array}{ccc} \lambda &	\ &  0 \\ 	& 	& \\ 0  &	\ &  \mu \end{array}\right)  = 
\left(\begin{array}{ccc} u_{00}^\ddag & \ & u_{10}^\ddag \\	& & 	\\  u_{01}^\ddag & \	&  u_{11}^\ddag \end{array}\right) 
 \left( \begin{array}{ccc} \alpha &	\ &  \beta \\ 	& & \\ 	\gamma &	\ &  \delta \end{array}\right)
\left(\begin{array}{ccc} u_{00} &	\ &  u_{01} \\ 	& & 	\\ u_{10} &	 \ &  u_{11} \end{array}\right) 
\]
\end{definition}

We now characterise when the matrix representation of an arrow (w.r.t. a certain self-similar structure) may be diagonalised by another self-similar structure:

\begin{theorem}
Let $(S,\code,\decode)$ and $(S,c,d)$  be self-similar structures of $(\PI,\uplus)$ at the same object, giving rise to strong embeddings $\phi_\dc , \phi_{(c,d)} : P_2\rightarrow \PI(S,S)$ and (equivalently)  internalisations of the disjoint union 
\[ \_ \uplus_\dc \_  ,  \_ \uplus_{(c,d)} \_ : \PI(S,S)\times \PI(S,S)\rightarrow \PI(S,S) \]
The matrices  representations that may be diagonalised by  the unique isomorphism between $(S,\code,\decode)$ and $(S,c,d)$ are exactly those of the form
\[ \left( \begin{array}{ccc} \phi_\dc(p)(X\uplus_{(c,d)} Y) \phi_\dc(p^{\ddag}) & 	\	\	& \phi_\dc(p)(X\uplus_{(c,d)} Y) \phi_\dc(q^{\ddag}) \\
							& 	\	\	& 					\\
				\phi_\dc(q)(X\uplus_{(c,d)} Y) \phi_\dc(p^{\ddag}) & 	\	\	&  \phi_\dc(q)(X\uplus_{(c,d)} Y) \phi_\dc(q^{\ddag}) \end{array} \right)
\]
\end{theorem}
\begin{proof} {\em In the following proof, we abuse notation slightly for purposes of clarity. We will denote 
\[ \phi_\dc(p), \phi_\dc(q) , \phi_\dc(p^\ddag) ,\phi_\dc(q^\ddag)\in \PI(S,S) \] by $p,q,p^\ddag,q^\ddag\in \PI(S,S)$, 
and similarly, denote 
\[  \phi_{(c,d)}(p), \phi_{(c,d)}(q) , \phi_{(c,d)}(p^\ddag) ,\phi_{(c,d)}(q^\ddag)\in \PI(S,S) \]  by $r,s,r^\ddag,s^\ddag\in \PI(S,S)$.}

Given arbitrary arrows $X,Y\in \PI(S,S)$, then  $[X\uplus_{(c,d)} Y]_{(c,d)}  =  \left( \begin{array}{cc} X & 0 \\ 0 & Y \end{array} \right)$, and all diagonal matrix representations (w.r.t. $(S,c,d)$) are of this form.  Let us now conjugate such a diagonal matrix by inverse of the matrix $U$ derived from Theorem \ref{changebasis}; this gives $U^{-1}[X\uplus_{(c,d)}Y]_{(c,d)} U =$
\[  \left( \begin{array}{ccc} pr^{-1}Xrp^{-1} \vee qr^{-1} Y rq^{-1} &	\	\	&  pr^{-1}X r q^{-1} \vee ps^{-1} Y sq^{-1} \\
		&		&		\\
							qr^{-1}Xrp^{-1} \vee qs^{-1}Ysp^{-1} &	\ 	\	& 	 qr^{-1} X rq^{-1} \vee qs^{-1} Y sq^{-1} 
\end{array}\right) \]
Comparing this with the explicit form of the internalisation of the disjoint union by the self-similar structure $(S,c,d)$ gives 
\[ \left( \begin{array}{ccc} p(X\uplus_{(c,d)} Y) p^{-1} &	\ \ &  p(X\uplus_{(c,d)} Y) q^{-1} \\
			&		&				\\
				q(X\uplus_{(c,d)} Y) p^{-1} &	\ \ 	&  q(X\uplus_{(c,d)} Y) q^{-1} \end{array} \right)
\]
Therefore all matrices of this form are diagonalised by the unique isomorphism from $(S,\code,\decode)$ to $(S,c,d)$.   Conversely, as $X,Y\in \PI(S,S)$ were chosen arbitrarily, all matrices diagonalised by this unique isomorphism are of this form.
\end{proof}

\begin{remark} It is worth emphasising that the above theorem characterises those matrix representations that may be diagonalised by a particular self-similar structure; it does not address the question of whether there exists a self-similar structure that diagonalises a particular matrix representation.  For the particular example of $\mathbb N$ as a self-similar object, an arrow $f\in \PI(\mathbb N,\mathbb N)$ is diagonalisable iff there exists a partition of $\mathbb N$ into disjoint infinite subsets $A\cup B= \mathbb N$ such that $f(A)\subseteq A$ and $f(B)\subseteq B$. Simple cardinality arguments will demonstrate that this question is undecidable in general.
\end{remark}

\section{Conclusions}
If nothing else, this paper has hopefully demonstrated that, although superficially dissimilar, the notions of copying derived from quantum mechanics and from logic (and categorical linguistics) are nevertheless closely connected. However, these connections are not apparent unless we allow for the definitions in both cases to be taken up to canonical isomorphisms. 

\section{Acknowledgements}
The author wishes to acknowledge useful discussions, key ideas, and constructive criticism from a variety of sources, with particular thanks to S. Abramsky, B. Coecke, C. Heunen, M. Lawson, P. Panangaden, and P. Scott.

 \bibliographystyle{plain}
\bibliography{lambek_refs}

\appendix

\section{Relating coherence for symmetric $\dagger$ semi-monoidal categories with coherence for internalised tensors}\label{simplecoherence}
Let  $(S,\code,\decode)$ be a $\dagger$ self-similar structure of a symmetric $\dagger$ semi-monoidal category $(\C,\otimes, \tau_{\_,\_,\_},\sigma_{\_,\_},(\ )^\dagger)$, and let the internal tensor and induced  canonical isomorphisms be denoted $\_ \otimes_\dc \_ :\C(S,S)\times\C(S,S)  \rightarrow \C(S,S)$ and $\tau_\dc,\sigma_\dc\in \C(S,S)$.  We consider the question of when a categorical diagram built inductively from the following toolkit 
\[ \{ S\in Ob(\C) , \_\otimes\_ , \tau_{\_,\_,\_} , \sigma_{\_ ,\_} , \_ \otimes_\dc \_ , \tau_\dc , \sigma_\dc , (\ )^\dagger \} \]
may be guaranteed to commute?

We will exhibit a sufficient condition for coherence of diagrams built inductively from this toolkit  that is closely related to MacLane's coherence theorems for associativity and symmetry \cite{MCL}, and the coherence theorem for self-similarity given in \cite{PHarxiv}, but first give some some preliminary definitions and results:

\begin{definition}
We define $Tree_S$ to be the set of free non-empty binary trees over a single symbol $S$ inductively, as follows:
\begin{itemize}
\item $S\in Tree_S$,
\item Given $A,B\in Tree_S$, then $(A\Box B)\in Tree_S$.
\end{itemize}
We will write members of $Tree_S$ as bracketed strings in the obvious way, such as $(S\Box((S\Box S)\Box S)\in Tree_S$.
\end{definition}

There exists an obvious map that assigns objects of $\C$ to members of $Tree_S$ given by instantiating each occurrence of the formal symbol $\Box$, as follows:

\begin{definition}\label{inst}
Given free non-empty binary tree $X\in Tree_S$, we define an object $Inst(X)\in Ob(\C)$. We give this assignment inductively:
\begin{itemize}
\item $Inst(S)=S\in Ob(\C)$.
\item $Inst(A\Box B) = Inst(A) \otimes Inst(B)\in Ob(\C)$
\end{itemize}
Note that the same object of $\C$ may be assigned to distinct members of $\C$ by this definition -- consider the special case where $(\C,\otimes)$ is strictly associative. See \cite{} for other examples not based on strict associativity. 
\end{definition}

We now give a symmetric $\dagger$ semi-monoidal category whose objects are free binary trees over $S$, and whose hom-sets are copies of hom-sets of $\C$. 

\begin{definition}
We define the symmetric $\dagger$ semi-monoidal category $(\F_S,\Box)$ as follows: 
\begin{itemize}
\item {\bf (Objects)} $Ob(\F_S)=Tree_S$.
\item {\bf (Arrows)} $\F_S(X,Y)=\C(Inst(X),Inst(Y)$
\item {\bf (Composition)} This is inherited from $\C$ in the obvious way. 
\item {\bf (Tensor)} 
\begin{itemize}
\item {\bf (On Objects)} Given $A,B\in Ob(F_S)$, then their tensor is the formal tree $A\Box B$.
\item {\bf (On Arrows)} Given $f\in \F_S(A,B)$ and $g\in \F_S(X,Y)$, then 
\[ f\Box g \ = \ f\otimes g\in \C(Inst(A)\otimes Inst(X),Inst(B)\otimes Inst(Y))  \]
Note that $\C(Inst(A)\otimes Inst(X),Inst(B)\otimes Inst(Y)) =$ \[  \C(Inst(A\Box X),Inst(B\Box Y)) = \F_S(A\Box X,B\Box Y) \] as required.
\end{itemize}
\item {\bf (Canonical Isomorphisms)} 
\begin{itemize}
\item Given $A,B,C\in Ob(\F_S)$, the associativity isomorphism 
\[ \tau_{A,B,C}\in \F_S(A\Box(B\Box C),(A\Box B)\Box C) \] 
is given by 
\[ \tau_{A,B,C}=\tau_{Inst(A),Inst(B),Inst(C)}\in \C(Inst(A\Box(B\Box C)),Inst((A\Box B)\Box C) \]
\item Given $A,B,\in Ob(\F_S)$, the symmetry isomorphism $\sigma_{A,B}\in \F_S(A\Box B),(B\Box A)$ is given by 
\[ \sigma_{A,B}=\sigma_{Inst(A),Inst(B)}\in \C(Inst(A\Box B),Inst(B\Box A)) \]
\end{itemize}
\item {\bf (Dagger Operation)} Given $X\in Ob(\F_S)$, we define, on objects, $X^\dagger =X$. On arrows, the dagger is inherited from $\C$, so given $f\in \F_S(X,Y)$, 
\[ f^\dagger = f^\dagger\in \C(Inst(Y),Inst(X)) = \F_S(Y,X) \]
\end{itemize}
\end{definition}

The symmetric $\dagger$ semi-monoidal category $(\F_S,\Box,(\ )^\dagger)$ may be thought of as a `well-behaved' version of the subcategory of $(\C,\otimes)$ generated by $S\in Ob(\C)$ --- in particular, all canonical (for associativity) diagrams over $(\F_S,\Box)$ commute (we refer to \cite{} for a proof of this, and a discussion of why this is not true in arbitrary monoidal or semi-monoidal categories). 

The $Inst$ mapping of Definition \ref{inst} lifts to a $\dagger$ semi-monoidal functor from $(\F_S,\Box)$ to $(\C,\otimes)$ in the obvious way:

\begin{definition}
We define 
$Inst: (\F_S,\Box,(\ )^\dagger)\rightarrow (\C,\otimes ,(\ )^\dagger)$ as follows:
\begin{itemize}
\item {\bf (Objects)}For all $T\in Ob(\F_S)=Tree_S$, we take $Inst(T)$ to be as in Definition \ref{inst}.
\item {\bf (Arrows)} Given $f\in \F_S(X,Y)$, we define 
\[ Inst(f)=f\in \C(Inst(X),Inst(Y))=\F_S(X,Y) \]
\end{itemize}
By contrast with the  familiar``ioof's'' ({\em Identity on Objects Functors}), this functor is the identity on arrows (i.e. homsets). Given this fact, it is straightforward to prove that $Inst$ is not only functorial, but preserves both the (semi-) monoidal structure and the dagger operation.
\end{definition}

As well as the above $\dagger$ monoidal functor, there exists another dagger monoidal functor to the endomorphism monoid  $\C(S,S)$, equipped with the  internalised tensor $\_ \otimes_\dc \_$. This is based on the following inductively defined arrows:
\begin{definition}
For all objects $X\in Ob(\F_S)$, we define the {\bf generalised code / decode} arrows $\code_X\in \F(X,S)$ and $\decode_X\in \F(S,X)$ 
inductively, as follows
\begin{itemize}
\item $\code_S=1_S\in \F_S(S,S)$
\item $\code_{A\Box B}=\code(\code_A\Box \code_B) \in \F_S(A\Box B,S)$
\item $\decode_X=\code_X^\dagger$
\end{itemize}
\end{definition}

The following definition is an extension of a construction of \cite{PH98,PHarxiv} to the $\dagger$ commutative setting. 
\begin{definition}
The {\bf generalised convolution functor} $\Phi_\dc:(\F_S,\Box, (\ )^\dagger )\rightarrow (\C(S,S),\otimes_\dc, ( \ )^\dagger )$ is defined by:
\begin{itemize}
\item {\bf (Objects)} $\Phi_\dc(X)=S$, for all $X\in Ob(\F_S)$.
\item {\bf (Arrows)} $\Phi(f)=\code_Yf\decode_X$ for all $f\in \F_S(X,Y)$.
\end{itemize}
Since both $\code_Y$ and $\decode_X$ are unitary, for all $X,Y\in Ob(\F_S)$,  this functor is fully faithful. A simple inductive argument (given in \cite{PHarxiv}) demonstrates that it is also a (semi-)monoidal functor, and thus, for all $X,Y,Z\in Ob(\F_S)$, 
\[ \Phi_\dc(\tau_{X,Y,Z})= \tau_\dc \ \ , \ \ \Phi_\dc(\sigma_{X,Y}) = \sigma_\dc \]
To see that it also preserves the dagger, note that 
\[ \Phi_\dc (f)^\dagger = \left( \code_Y f \decode_X\right)^\dagger = \code_X f^\dagger \decode_Y = \Phi_\dc\left( f^\dagger \right) \]
\end{definition}

We now use the toolkit developed to give a sufficient (albeit non-constructive) condition for commutativity of `mixed' diagrams built inductively from 
\[  \{  S\in Ob(\C) \ ,\  \_\otimes\_ \ ,\  \tau_{\_,\_,\_} \ ,\  \sigma_{\_ ,\_} \ ,\  \_ \otimes_\dc \_ \ ,\  \tau_\dc \ ,\  \sigma_\dc \ ,\  (\ )^\dagger \} \]
\begin{lemma}\label{abovebelow}
Let $\mathfrak D$ be an arbitrary diagram over $\F_S$. Then 
\begin{enumerate}
\item $\mathfrak D$ commutes $\Rightarrow$  $Inst(\mathfrak D)$ commutes.
\item $\mathfrak D$ commutes  $\Leftrightarrow$ $\Phi_\dc (\mathfrak D)$ commutes.
\end{enumerate}
\end{lemma}
\begin{proof}
$ $ \\
\begin{enumerate}
\item This is a simple consequence of the fact that functors preserve commutativity of diagrams.
\item The $\Rightarrow$ implication is trivial. The opposite direction is due to the fact that $\Phi_\dc$ is fully faithful; consider a diagram $\mathfrak D$ over $\mathcal F_S$ such as 
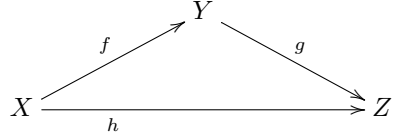
\begin{center}
$ \xymatrix{ 
																					&& Y \ar[drr]^{g}				&&													\\
X \ar[rrrr]_<<<<<<<<<<{h} \ar[urr]^f					&&																	&& Z				\\
}
$
\end{center}
Applying $\Phi_\dc$ to this diagram gives $\Phi_\dc(\mathfrak D)$ as follows:
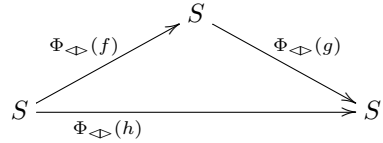
\begin{center}
$ \xymatrix{ 
																					&& S	 \ar[drr]^{\Phi_\dc(g)}			&& 													\\
S  \ar[rrrr]_<<<<<<<<<<{\Phi_\dc(h)} \ar[urr]^{\Phi_\dc(f )}		&&																	&&  S
}
$
\end{center}
and by definition of $\Phi_\dc$, the following diagram commutes if and only if $\mathfrak D$ commutes, if and only if $\Phi_\dc(\mathfrak D)$ commutes:
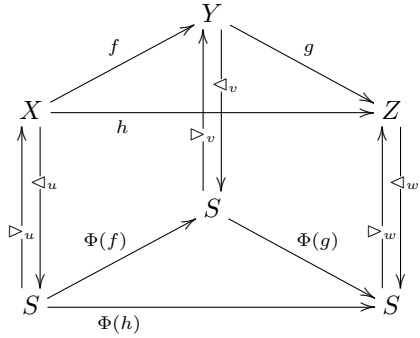
\begin{center}
$ \xymatrix{ 
																			&& Y \ar[drr]^{g}  \ar@<0.8ex>[dd]|<<<<<<<<{\  \code_v}					&&													\\
X \ar[rrrr]_<<<<<<<<<<{h} \ar[urr]^f \ar@<0.8ex>[dd]|<<<<<<<<{\  \code_u}					&&																	&& Z	 \ar@<0.8ex>[dd]|<<<<<<<<{\  \code_w}			\\
																					&& S	 \ar[drr]^{\Phi(g)}\ar@<0.8ex>[uu]|<<<<<<<<{ \decode_v}				&& 													\\
S  \ar[rrrr]_<<<<<<<<<<{\Phi(h)} \ar[urr]^{\Phi(f )} \ar@<0.8ex>[uu]|<<<<<<<<{ \decode_u}		&&																	&&  S \ar@<0.8ex>[uu]|<<<<<<<<{ \decode_w}	
}
$
\end{center}
The generalisation of the above reasoning to arbitrary diagrams is then immediate.
 \end{enumerate}
 \end{proof}
 
 The above lemma allows us to reduce the question posed at the beginning of this Appendix to a simpler question that may be resolved by an appeal to well-established theory.  
 
 \begin{proposition}\label{freecommuting}
 Let $\mathcal M$ be a diagram over $\mathcal F_S$ built from the following toolkit:
 \[ \{ S\in Ob(\F_S) , \_\Box\_ , \tau_{\_,\_,\_} , \sigma_{\_ ,\_} , \_ \otimes_\dc \_ , \tau_\dc , \sigma_\dc , (\ )^\dagger \} \]
 The question of whether $\Phi_\dc(\mathfrak M)$ may be guaranteed to commute may be dealt with by an appeal to MacLane's coherence theorem for symmetry and associativity \cite{MCL}.
 \end{proposition}
 \begin{proof} Recall from Part 2. of Lemma \ref{abovebelow} that  $\Phi_\dc(\mathfrak M)$ commutes if and only if $\mathfrak M$ commutes. Now recall that (almost by definition), for all $X,Y,Z\in Ob(\F_S)$,
 \[ \Phi_\dc(\tau_{X,Y,Z}) = \tau_\dc \ , \ \Phi_\dc(\sigma_{X,Y}) = \sigma_\dc   \]
 Similarly, $\Phi_\dc$ acts as the identity on the endomorphism monoid $\F_S(S,S)$, so 
 \[ \Phi_\dc(\tau_\dc) = \tau_\dc \ , \ \Phi_\dc(\sigma_\dc) = \sigma_\dc   \]
 Finally, since  $\Phi_\dc(f\otimes g)  =\Phi_\dc(f)\otimes_\dc \Phi_\dc(g)$, we may conclude that $\Phi_\dc(\mathfrak M)$ is a canonical diagram for a symmetric monoidal tensor, built from the toolkit
 \[ \{ \_ \otimes_\dc \_ \ , \ \tau_\dc \ , \ \sigma_\dc \} \]
 We have thus reduced the question of commutativity of $\mathcal M$ to a case that may be satisfied by appealing to MacLane's theorems for symmetry and associativity.
 \end{proof}
 
 This does not (yet) quite answer the question posed at the beginning of this Appendix. Recall that $(\F_S,\Box)$ is the monoidal category freely generated by $S\in Ob(\C)$, rather than the subcategory of $(\C,\otimes) $ generated by $S$.  Instead, we require the following result:
  \begin{theorem}
 Let $\mathfrak M$ be a diagram over $\C$ built from the following toolkit:
 \[ \{ S\in Ob(C) , \_\otimes\_ , \tau_{\_,\_,\_} , \sigma_{\_ ,\_} , \_ \otimes_\dc \_ , \tau_\dc , \sigma_\dc , (\ )^\dagger \} \]
Then $\mathfrak M$ is guaranteed to  commute when there exists some canonical (for associativity and commutativity) diagram $\mathfrak C$ over $(\C(S,S),\otimes_\dc)$ that is guaranteed to commute by MacLane's coherence theorems, together with 
some diagram $\mathfrak T$ over $\F_S$, satisfying 
\[ Inst(\mathfrak T) = \mathfrak M\ \ \mbox{ and } \ \ \Phi_\dc(\mathfrak T) =\mathfrak C \]
\end{theorem}
\begin{proof}
This follows from Proposition \ref{freecommuting} and Lemma \ref{abovebelow} above. Note that a suitable diagram $\mathfrak T$, when it exists, need not be unique.
\end{proof}

\begin{remark}
The above theorem is non-constructive, in that it is based on the pure existence of some appropriate diagram $\mathfrak T$ over $\F_S$. However, starting from the diagram $\mathfrak M$, it is not difficult to conceive of a simple algorithm that will allow us to decide whether or not such a diagram $\mathfrak T$ exists, and provide a concrete example when this is the case. 
\end{remark}

\end{document}